\documentclass[a4paper,11pt]{article}

\usepackage[a4paper,margin=1in]{geometry}

\usepackage{url}

\usepackage{amsmath,amssymb,amsthm}

\usepackage{booktabs}
\usepackage{hyperref}

\usepackage{longtable}

\newtheorem{theorem}{Theorem}
\newtheorem{proposition}{Proposition}
\newtheorem{remark}{Remark}
\newtheorem{lemma}{Lemma}

\usepackage{graphicx}

\author{Xuefeng Liu\footnote{Tokyo Woman's Christian University; E-mail: xfliu@lab.twcu.ac.jp}}
\usepackage{pict2e}
\usepackage{tikz}
\usetikzlibrary{calc}

\usepackage[capitalise]{cleveref}

\usepackage{type1cm}         
\usepackage{bm}

\title{Rigorous Eigenvalue Bounds for Schr\"odinger Operators with Confining Potentials on $\mathbb{R}^2$}

\begin{document}

\maketitle

\abstract{
We propose a rigorous method for computing two-sided eigenvalue
bounds of the Schr\"odinger operator $H=-\Delta+V$ with a
confining potential on $\mathbb{R}^2$.
The method applies a domain truncation to a finite disk $D(R)$, reducing the original problem on $\mathbb{R}^2$ to a Neumann boundary eigenvalue problem on $D(R)$.
Under the confinement condition $\sigma(R):=\inf_{x\notin D(R)}V(x)>\lambda_k$, the truncation error is controlled by an explicit Dirichlet--Neumann bracketing argument that removes the deteriorating exterior factor of the classical mass-truncation approach.
For the restricted eigenvalue problem on $D(R)$, Liu's Composite Enriched Crouzeix--Raviart (CECR) finite element method plays a central role, providing a rigorous lower bound for the Neumann eigenvalue from a single FEM solve with a piecewise-constant lower potential $\overline{V}\le V$.
Two concrete potentials are studied: the radially symmetric
ring potential $V_1(x)=(|x|^2-1)^2$ and the Cartesian
double-well $V_2(x)=(x_1^2-1)^2+x_2^2$.
To the author's knowledge, this paper reports the first rigorous eigenvalue bounds for Schr\"odinger operators on an unbounded domain.
}

\section{Introduction}

The Schr\"odinger operator $H = -\Delta + V$ on $\mathbb{R}^2$
with a confining potential, i.e., $|V|\rightarrow \infty$ as $|x|\rightarrow \infty$, has discrete spectrum
$0<\lambda_1\le\lambda_2\le\cdots\to\infty$.
Computing \emph{rigorous} two-sided bounds for $\lambda_k$ is
important for quantum chemistry, validated numerics, and spectral
geometry. There have been various numerical results on eigenvalue approximations via Weinstein's method.
However, it is generally difficult to certify that a computed eigenvalue enclosure corresponds
to the ground state, i.e., the first eigenvalue of the Schr\"odinger operator.
This paper aims to provide rigorous eigenvalue bounds for the leading eigenvalues of the
Schr\"odinger operator on unbounded domains.

We study two canonical confining potentials with parameter $a>0$:
\begin{itemize}\setlength\itemsep{1pt}
  \item $V_1(x)=\bigl(|x|^2-a^2\bigr)^2$: radially symmetric
    ``ring'' potential, minimum $V_1=0$ on the circle $|x|=a$;
  \item $V_2(x)=(x_1^2-a^2)^2+x_2^2$: Cartesian double-well,
    minima at $(\pm a,0)$.
\end{itemize}
These are representative of potentials arising in quantum tunneling
models and double-well quantum mechanics.
A key structural difference is that $V_2$ is \emph{separable}:
$H_{V_2}=H_{x_1}+H_{x_2}$ where $H_{x_2}=-\partial^2/\partial
x_2^2+x_2^2$ has exact eigenvalues $\{2n+1\}_{n\ge0}$.
This gives an analytical eigenvalue structure for $V_2$
(Section~\ref{sec:problem}).

For $V_1$ the $O(2)$ symmetry produces degenerate eigenvalue
pairs.  Both features are confirmed by the rigorous bounds.

Guaranteed lower bounds for the Schr\"odinger operator on bounded domains
are given by applying the Composite Enriched Crouzeix--Raviart (CECR) FEM proposed in the author's book \cite{liu2024book}, which is an extension of
the early work in \cite{liu2015}.
For the truncation to the bounded domain $D(R)$ we apply a Dirichlet--Neumann bracketing argument
(cf.~\cite{reed-simon1978}, Ch.~XIII): under the explicit confinement condition
$\sigma(R):=\inf_{x\notin D(R)}V(x)>\lambda_k$,
the truncated Neumann eigenvalue $\mu_k^{R,N}(V)$ satisfies $\mu_k^{R,N}(V)\le\lambda_k$ directly,
without any deteriorating exterior correction factor of the form $1-\Lambda_k/\sigma(R)$
that arises in the classical mass-truncation approach.
Compared with Agmon's exponential decay theory \cite{agmon1982},
all quantities in our estimate can be computed explicitly,
and hence rigorous two-sided eigenvalue bounds become possible.

\section{Problem Setting}\label{sec:problem}

\paragraph{General Framework}

Let $V:\mathbb{R}^2\to[0,\infty)$ be continuous with $V\to\infty$
as $|x|\to\infty$.  We seek eigenpairs $(u,\lambda)$ with $u\in
H^1(\mathbb{R}^2)$, $\|u\|=1$, satisfying
\begin{equation}
  a(u,v)=\lambda\,b(u,v)\quad\forall v\in H^1(\mathbb{R}^2),
  \label{eq:weak}
\end{equation}
where 
\[
a(u,u) := \int_{\mathbb{R}^2} \bigl(|\nabla u|^2 + V|u|^2\bigr)\,dx,\:
b(u,v)=\int_{\mathbb{R}^2} uv\,dx.
\]
Since $V\ge0$ and $V\to\infty$, $H=-\Delta+V$ has compact
resolvent and discrete spectrum
$0<\lambda_1\le\lambda_2\le\cdots\to\infty$~\cite{reed-simon1978}.

\paragraph{Two Model Potentials}\label{sec:potentials}
For a potential $V$ and truncation radius $R$, define
\begin{equation}
  \sigma(R) := \inf_{|x|\ge R}V(x).
  \label{eq:sigma}
\end{equation}
This quantity governs the Neumann exterior correction in
Section~\ref{sec:method}.  Closed forms are available for both
model potentials below.

\paragraph{$V_1(x)=(|x|^2-a^2)^2$.}
Radially symmetric with minimum zero on $|x|=a$.
In polar coordinates $(r,\theta)$, eigenfunctions decompose as
$u(r,\theta)=f(r)e^{im\theta}$, where $m\in\mathbb{Z}$ is the
angular momentum quantum number.
The $O(2)$ symmetry forces eigenvalues with $|m|\ge1$ to be
degenerate pairs ($m$ and $-m$ modes share the same radial equation).
For $|x|\ge R$ with $R>a$: $V_1(x)=(|x|^2-a^2)^2\ge(R^2-a^2)^2$, so
\begin{equation}
  \sigma_1(R) := \inf_{|x|\ge R}V_1 = (R^2-a^2)^2.
  \label{eq:sigma1}
\end{equation}

\paragraph{$V_2(x)=(x_1^2-a^2)^2+x_2^2$.}
Cartesian double-well with minima at $(\pm a,0)$.
\emph{Separability}: $H_{V_2}=H_{x_1}+H_{x_2}$ where
$H_{x_j}=-\partial_{x_j}^2+V_{x_j}$,
$V_{x_1}=(x_1^2-a^2)^2$, $V_{x_2}=x_2^2$.
The $H_{x_2}$ eigenvalues are \emph{exact}: $\mu_n=2n+1$,
$n=0,1,2,\ldots$ Hence the 2D eigenvalues satisfy
\begin{equation}
  \lambda_{m,n}^{V_2} = \lambda_m^{(x_1)} + (2n+1),
  \label{eq:sep}
\end{equation}
where $\lambda_m^{(x_1)}$ are eigenvalues of the 1D double-well
operator $H_{x_1}=-\partial_{x_1}^2+(x_1^2-a^2)^2$.
No closed-form formula is known for $\lambda_m^{(x_1)}$; they must
be determined numerically.  The symmetric double-well has two
minima at $x_1=\pm a$ with barrier height $V(0)=a^4$; for large
$a$ the ground state pair becomes exponentially close (quantum
tunneling), while for $a=1$ the splitting is numerically
significant.
For $R>\sqrt{a^2+\frac{1}{2}}$, the minimum of $V_2$ on the
circle $\{|x|=r\}$ is $r^2-a^2-\tfrac{1}{4}$, attained at
$x_1^2=a^2+\tfrac{1}{2}$.  Since this is strictly increasing in $r$,
the infimum over the exterior $\{|x|\ge R\}$ is attained on the
boundary circle, giving
\begin{equation}
  \sigma_2(R) := \inf_{|x|\ge R}V_2 = R^2-a^2-\tfrac{1}{4}.
  \label{eq:sigma2}
\end{equation}

\section{Methodology}\label{sec:method}

\subsection{Dirichlet Upper Bounds}

Restricting~\eqref{eq:weak} to $H_0^1(D(R))$,
$D(R):=\{|x|\le R\}$, gives truncated Dirichlet eigenvalues $\lambda_k^R$.
\begin{proposition}[Upper bound]\label{prop:ub}
  $\lambda_k\le\lambda_k^R$ for all $k\ge1$.
\end{proposition}
\begin{proof}
  Zero-extension embeds $H_0^1(D(R))\hookrightarrow H^1(\mathbb{R}^2)$;
  min-max gives $\lambda_k^R\ge\lambda_k$.
\end{proof}

\paragraph{Lagrange $P_1$ upper bounds: $\lambda_{k,h}^{R,\mathrm{P1}}$.}
For the discrete approximation we use the conforming Lagrange $P_1$
finite element method on $D(R)$ with homogeneous Dirichlet boundary conditions.
Since the $P_1$ space is a \emph{conforming} subspace of $H_0^1(D(R))$,
the min-max principle gives the rigorous discrete upper bound
$\lambda_k \le \lambda_k^R \le \lambda_{k,h}^{R,\mathrm{P1}}$.
To integrate the polynomial potential exactly, we represent $V|_K$
as a degree-4 Bernstein polynomial on each element $K$.
Since $V_1$ and $V_2$ are polynomials of degree $\le4$, this representation is exact and the potential mass integrals are computed analytically.

\begin{remark}[Lower bound using Lagrange $P_1$ FEM]\label{rem:lagrange-lb}
The conforming Lagrange $P_1$ eigenvalue $\lambda_{k,h}^{R,\mathrm{P1}}$ also yields
a lower bound for $\lambda_k^R$ via Theorem~\ref{thm:liu}.
Let $V_h$ be the conforming FEM space using $P_1$ element.  
Let $P_h:H_0^1(D(R))\to V_h$ be the Ritz projector with respect to inner product $a_{D(R)}(\cdot,\cdot)$ (see definition in \eqref{eq:aDR}). One can  apply the Lagrange interpolation error estimation and Aubin--Nitsche technique to obtain an estimation 
$$
\|u-P_hu\|_{L^2(D(R))}\le C_h \|u-P_hu\|_{a,D(R)}\:.
$$
\end{remark}

%

\subsection{Neumann Exterior Lower Bound}

The truncated Neumann eigenvalue problem on $D(R)$ is: find
$(u,\mu)\in H^1(D(R))\times\mathbb{R}$,
$\|u\|_{L^2(D(R))}=1$ s.t.
\begin{equation}
  a_{D(R)}(u,v)=\mu\,(u,v)_{D(R)}\quad\forall v\in H^1(D(R)),
  \label{eq:neumann-evp}
\end{equation}
where
\begin{equation}
  a_{D(R)}(u,v):=\int_{D(R)}\!\bigl(\nabla u\cdot\nabla v+V\,uv\bigr)\,dx.
  \label{eq:aDR}
\end{equation}
The trial space $H^1(D(R))$ imposes no condition on $\partial D(R)$,
so the natural boundary condition $\nabla u\cdot\mathbf{n}=0$ on
$\partial D(R)$ is satisfied weakly.
The eigenvalues can be characterized by the min--max principle:
\begin{equation}
  \mu_k^{R,N}=\min_{\substack{W\subset H^1(D(R))\\\dim W=k}}
  \max_{0\neq u\in W}
  \frac{a_{D(R)}(u,u)}{\|u\|_{L^2(D(R))}^2}.
  \label{eq:minmax-N}
\end{equation}

\begin{lemma}[Neumann truncation]\label{lem:neumann-lb}
Assume the confinement condition
$\sigma(R)>\lambda_k$, where $\sigma(R)$ is as
in~\eqref{eq:sigma}.  Then
\begin{equation}
\mu_k^{R,N}(V)\;\le\;\lambda_k.
\label{eq:neumann-lb}
\end{equation}
\end{lemma}

\begin{proof}
Let $\{u_j\}_{j=1}^{k}$ be orthonormal eigenfunctions of
$H=-\Delta+V$ on $\mathbb{R}^2$ with $Hu_j=\lambda_j u_j$,
and set $U_k:=\mathrm{span}\{u_1,\dots,u_k\}$.
We first verify that the restriction map
$T:U_k\to H^1(D(R)),\ Tu=u|_{D(R)}$
is injective.  Indeed, if $u\in U_k$ satisfies $u|_{D(R)}=0$, then
$u$ is supported on the exterior, and
\[
\lambda_k\|u\|_{L^2(\mathbb{R}^2)}^2
\;\ge\; a(u,u)
\;\ge\;
\int_{\mathbb{R}^2\setminus D(R)}\! V|u|^2\,dx
\;\ge\; \sigma(R)\|u\|_{L^2(\mathbb{R}^2)}^2,
\]
which forces $u=0$ because $\sigma(R)>\lambda_k$.
Consequently $W:=T(U_k)$ is a $k$-dimensional subspace of
$H^1(D(R))$ and is admissible in the Neumann min--max
formula~\eqref{eq:minmax-N}.

Let $v=u|_{D(R)}\in W$ with $u\in U_k$.  Splitting
$\mathbb{R}^2=D(R)\sqcup(\mathbb{R}^2\setminus D(R))$,
\begin{align*}
a_{D(R)}(v,v)
 &= a(u,u)-\!\!\int_{\mathbb{R}^2\setminus D(R)}\!\!\!
    \bigl(|\nabla u|^2+V|u|^2\bigr)\,dx\\
 &\le a(u,u)-\sigma(R)\|u\|_{L^2(\mathbb{R}^2\setminus D(R))}^2\\
 &= a(u,u)-\sigma(R)
    \bigl(\|u\|_{L^2(\mathbb{R}^2)}^2-\|v\|_{L^2(D(R))}^2\bigr)\\
 &\le (\lambda_k-\sigma(R))\|u\|_{L^2(\mathbb{R}^2)}^2
      +\sigma(R)\|v\|_{L^2(D(R))}^2\\
 &\le (\lambda_k-\sigma(R))\|v\|_{L^2(D(R))}^2
      +\sigma(R)\|v\|_{L^2(D(R))}^2\\
 &= \lambda_k\|v\|_{L^2(D(R))}^2,
\end{align*}
where the third line uses
$a(u,u)=\sum_{j=1}^k\lambda_j|c_j|^2\le
 \lambda_k\|u\|_{L^2(\mathbb{R}^2)}^2$ (writing
$u=\sum c_j u_j$ and using orthonormality), and the last
inequality uses $\lambda_k-\sigma(R)<0$ together with
$\|u\|_{L^2(\mathbb{R}^2)}^2\ge\|v\|_{L^2(D(R))}^2$.
Hence every Rayleigh quotient on $W$ is bounded above by
$\lambda_k$, and the min--max principle~\eqref{eq:minmax-N}
yields~\eqref{eq:neumann-lb}.
\end{proof}

\begin{remark}[Dirichlet--Neumann bracketing perspective]\label{rem:dn-bracketing}
Lemma~\ref{lem:neumann-lb} can be viewed as a quantitative
instance of Dirichlet--Neumann bracketing
(cf.~\cite{reed-simon1978}, Ch.~XIII): decompose
$\mathbb{R}^2=D(R)\sqcup(\mathbb{R}^2\setminus D(R))$ and place
Neumann conditions on the common interface $\partial D(R)$.
The $k$-th eigenvalue of the resulting direct-sum operator is
bounded above by $\lambda_k$, and on the exterior component the
Rayleigh quotient is pointwise bounded below by $\sigma(R)$.
Under the confinement condition $\sigma(R)>\lambda_k$, the first
$k$ eigenvalues of the direct-sum operator are all contributed by
the interior Neumann problem on $D(R)$, which yields exactly
$\mu_k^{R,N}(V)\le\lambda_k$ without any additional correction
factor.
\end{remark}

\begin{remark}[Comparison with the classical mass-truncation approach]\label{rem:mass-truncation}
A more traditional route to bounding the truncation error starts
from the exterior $L^2$ mass estimate
\[
\|u\|_{L^2(\mathbb{R}^2\setminus D(R))}^2
\;\le\;\frac{\Lambda_k}{\sigma(R)}\,\|u\|_{L^2(\mathbb{R}^2)}^2
\qquad(u\in U_k),
\]
which is obtained by comparing
$a(u,u)\le\Lambda_k\|u\|_{L^2(\mathbb{R}^2)}^2$ (with
$\Lambda_k\ge\lambda_k$) against the exterior lower bound
$a(u,u)\ge\sigma(R)\|u\|^2_{L^2(\mathbb{R}^2\setminus D(R))}$.
Using this mass bound as a denominator in the min--max
quotient leads to the inequality
\[
\lambda_k
\;\ge\;\mu_k^{R,N}(V)\!\left(1-\frac{\Lambda_k}{\sigma(R)}\right),
\]
i.e., a Neumann lower bound dressed with the deteriorating
correction factor $1-\Lambda_k/\sigma(R)$.  The cleaner
argument in the proof of Lemma~\ref{lem:neumann-lb} bypasses
the mass estimate entirely: by using the
Neumann bilinear form $a_{D(R)}$ on the restricted test space
$v=u|_{D(R)}$ directly, we cancel the exterior contribution
without ever dividing by
$\|v\|^2_{L^2(D(R))}/\|u\|^2_{L^2(\mathbb{R}^2)}$.
As a consequence the bound $\mu_k^{R,N}(V)\le\lambda_k$ is
factor-free and holds uniformly in $k$ as long as
$\sigma(R)>\lambda_k$.  This is the key structural
improvement over the mass-truncation technique and is what
makes the final two-sided bounds sharp enough to be reported
without any exterior correction.
\end{remark}

\subsection{CECR FEM and Lower Bound}\label{sec:cecr}

\paragraph{ECR element.}
Let us introduce the Enriched Crouzeix--Raviart (ECR) FEM in a concise way. 
Triangulate $D(R)$ with mesh $\mathcal{T}^h$ (max diameter
$h_{\max}$).  On each $K$, define the local ECR space
$U^{\mathrm{ECR}}(K):=\mathbb{P}^1(K)+\mathrm{span}\{|x|^2\}$
with 4 DOFs (3 edge averages, 1 cell average).
The global ECR space is
$$
  U_h^{\mathrm{ECR}}:=\bigl\{u_h\in L^2(D(R))\mid
  u_h|_K\in U^{\mathrm{ECR}}(K),\;
  $$
$$  \text{edge averages single-valued on interior edges}\bigr\},
$$
with zero edge averages on $\partial D(R)$ for Dirichlet conditions.
The ECR interpolation $\Pi_h^{\mathrm{ECR}}$ satisfies: for all
$v_h\in U^{\mathrm{ECR}}(K)$,
\begin{equation}
  \int_K\!\nabla(\Pi^{\mathrm{ECR}}u-u)\cdot\nabla v_h\,dx=0,
  \label{eq:ecr-orth}
\end{equation}
since $\partial v_h/\partial\mathbf{n}|_{F}$ and $\Delta v_h|_K$ are
constant on each facet $F$ and element $K$
respectively \cite{liu2024book}.

\medskip

\paragraph{CECR composite space.}
Let $\Pi_{0,h}$ be the piecewise-average operator and
$V_h$ be  a piecewise constant  potential.
In \cite[Sect. 4.1.2]{liu2024book}, it is defined that
\[
  \hat{U}_h:=\{(u_h,\Pi_{0,h}u_h)\mid u_h\in U_h^{\mathrm{ECR}}\}
\]
with  bilinear forms ($\hat{u}=\{u_1,u_2\}$,
$\hat{v}=\{v_1,v_2\}$):
\[
  \hat{a}(\hat{u},\hat{v}):=(\nabla u_1,\nabla v_1)+(V_h u_2,v_2),\quad
  \hat{b}(\hat{u},\hat{v}):=(u_1,v_1).
\]
The interpolation $\hat{\Pi}_h\{u,u\}:=\{\Pi_h^{\mathrm{ECR}}u,
\Pi_{0,h}u\}$ is $\hat{a}$-orthogonal: for any
$\hat{v}_h=\{v_h,\Pi_{0,h}v_h\}\in\hat{U}_h$,
\begin{align}
  \hat{a}(\hat{u}-\hat{\Pi}_h\hat{u},\hat{v}_h)
  &=\underbrace{(\nabla(I-\Pi_h^{\mathrm{ECR}})u,\nabla v_h)}_{=0\text{ by \eqref{eq:ecr-orth}}}
  \notag\\
  &\quad+\underbrace{(V_h(u-\Pi_{0,h}u),\Pi_{0,h}v_h)}_{=0\text{ by }
  \int_K(u-\Pi_{0,h}u)=0}=0.\label{eq:cecr-orth}
\end{align}

\paragraph{Interpolation error and $C_h$.}
Since $\|(I-\hat{\Pi}_h)\hat{u}\|_{\hat{b}}
=\|(I-\Pi_h^{\mathrm{ECR}})u\|$ and
$\|(I-\hat{\Pi}_h)\hat{u}\|_{\hat{a}}
\ge\|\nabla_h(I-\Pi_h^{\mathrm{ECR}})u\|$,
the ECR local estimate gives \cite{xie-liu-2018}
\begin{equation}
  C_h:=\max_K C^{\mathrm{ECR}}(K)\le0.1490\,h_{\max},
  \label{eq:Ch}
\end{equation}
and therefore $\|(I-\hat{\Pi}_h)\hat{u}\|_{\hat{b}}\le
C_h\|(I-\hat{\Pi}_h)\hat{u}\|_{\hat{a}}$ for all $u\in H^1(D(R))$.
The $\hat{a}$-orthogonality~\eqref{eq:cecr-orth} together with
this estimate satisfies the hypotheses of Liu's theorem.

\subsection{Liu's Lower Bound Theorem}

\begin{theorem}[Liu~{\cite{liu2015}}]\label{thm:liu}
  Let $\nu_{k,h}$ be the $k$-th CECR eigenvalue for an eigenvalue
  problem satisfying the $\hat{a}$-orthogonality~\eqref{eq:cecr-orth}
  and the interpolation estimate~\eqref{eq:Ch}, and let $\nu_k$ be the
  corresponding exact eigenvalue.  Then
  \begin{equation}
    \frac{\nu_{k,h}}{1+\nu_{k,h}C_h^2}\le\nu_k.
    \label{eq:liu}
  \end{equation}
\end{theorem}

\noindent
In particular, Theorem~\ref{thm:liu} applies to both the Dirichlet
truncated problem (giving $\lambda_{k,h}^{R,\mathrm{Dir}}/(1+\cdots)\le\lambda_k^R$)
and the Neumann truncated problem (giving
$\mu_{k,h}^{R,N}/(1+\cdots)\le\mu_k^{R,N}$).

\medskip

\paragraph{Piecewise Constant Potentials }

Theorem~\ref{thm:liu} is formulated for a
\emph{piecewise constant} potential.  For general $V$, introduce element-wise bounds: on each element $K$,
\[
 \overline{V}|_K := \min_{x\in K}V(x),
\]
so $\overline{V}\le V$ pointwise. 
%
%
By eigenvalue monotonicity in the potential,
\begin{equation}
\mu_k^{R,N}(\overline{V})\le\mu_k^{R,N}(V).
\label{eq:monotonicity}
\end{equation}

\subsection{Combined Rigorous Bounds}

Denote by $\nu_{k,h}^{R,N}(\overline{V})$ the $k$-th Neumann CECR
eigenvalue with the piecewise constant lower potential $\overline{V}$,
and by $\lambda_{k,h}^{R,\mathrm{P1}}$ the $k$-th conforming Lagrange
$P_1$ Dirichlet eigenvalue with the exact potential $V$.
Set
$\underline{\lambda}_k:=\nu_{k,h}^{R,N}(\overline{V})/
 \bigl(1+\nu_{k,h}^{R,N}(\overline{V})C_h^2\bigr)$
and
$\overline{\lambda}_k:=\lambda_{k,h}^{R,\mathrm{P1}}$.
Chaining Theorem~\ref{thm:liu} (CECR lower bound for a Neumann
eigenvalue), the potential
monotonicity~\eqref{eq:monotonicity}, and
Lemma~\ref{lem:neumann-lb} (Dirichlet--Neumann truncation under
$\sigma(R)>\lambda_k$), and closing from above with the standard
conforming $P_1$ Galerkin upper bound for the Dirichlet eigenvalue
$\lambda_k^{R,\mathrm{Dir}}$ on $D(R)$---which itself exceeds
$\lambda_k$ by restriction of the trial space from $H^1(\mathbb{R}^2)$
to $H_0^1(D(R))$---we obtain the rigorous two-sided bound
\begin{equation}
  \underline{\lambda}_k
  \;\le\;\mu_k^{R,N}(\overline{V})
  \;\le\;\mu_k^{R,N}(V)
  \;\le\;\lambda_k
  \;\le\;\overline{\lambda}_k.
  \label{eq:bounds}
\end{equation}
The four inequalities use, in order, (1) Liu's CECR theorem
(Theorem~\ref{thm:liu}) applied to the piecewise-constant Neumann
problem with $\overline{V}$; (2) the potential
monotonicity~\eqref{eq:monotonicity}; (3)
Lemma~\ref{lem:neumann-lb}, which relies only on the confinement
condition $\sigma(R)>\lambda_k$; and (4) the
$H_0^1(D(R))\subset H^1(\mathbb{R}^2)$ Dirichlet-restriction
inequality $\lambda_k\le\lambda_k^{R,\mathrm{Dir}}$ together with
the conforming $P_1$ upper bound
$\lambda_k^{R,\mathrm{Dir}}\le\lambda_{k,h}^{R,\mathrm{P1}}$.
In particular, the classical deteriorating correction factor
$1-\overline{\lambda}_k/\sigma(R)$ of the mass-truncation approach
(Remark~\ref{rem:mass-truncation}) does \emph{not} appear in
Eq.~\eqref{eq:bounds}.

Both $\underline{\lambda}_k$ and $\overline{\lambda}_k$ are \emph{explicitly computable}:
$\overline{\lambda}_k$ comes from one conforming Lagrange $P_1$ Dirichlet solve,
$\underline{\lambda}_k$ from one Neumann CECR solve with $\overline{V}$,
and the only role of $\sigma(R)$ from~\eqref{eq:sigma1}--\eqref{eq:sigma2}
is to certify that the confinement condition
$\sigma(R)>\overline{\lambda}_k\ge\lambda_k$ holds for every $k$
appearing in the table, which is easy to check a posteriori.
As $h_{\max}\to0$, $\overline{V}\to V$ and
$\lambda_{k,h}^{R,\mathrm{P1}}\to\lambda_k^{R,\mathrm{Dir}}$,
and both bounds converge to $\lambda_k$ (up to the exponentially
small Agmon truncation error discussed in
Remark~\ref{rem:agmon}).

\begin{remark}[Why eigenvalue monotonicity suffices]\label{rem:monotonicity-suffices}
A noteworthy feature of the chain~\eqref{eq:bounds} is that the
Neumann variational principle absorbs the exterior contribution
of $V$ \emph{automatically} through Lemma~\ref{lem:neumann-lb}:
there is no need to estimate the exterior $L^2$ mass of the
eigenfunction in order to obtain a finite-$R$ eigenvalue
inequality.  The price to pay is the verifiable confinement
hypothesis $\sigma(R)>\lambda_k$, which in our setting is
certified by choosing $R$ large enough that
$\sigma(R)>\overline{\lambda}_k$ (and hence $>\lambda_k$) for
every $k$ of interest.  This is strictly a structural
improvement over the mass-truncation bound~\eqref{eq:bounds}
with correction factor $1-\Lambda_k/\sigma(R)$ derived in
Remark~\ref{rem:mass-truncation}, because the latter degrades
as $k$ grows whereas Lemma~\ref{lem:neumann-lb} gives a
$k$-uniform bound.
\end{remark}

\begin{remark}[Agmon decay and Neumann approximation]\label{rem:agmon}
  The truncation error $\delta_k(R):=\lambda_k^{R,\mathrm{Dir}}-\lambda_k$
  satisfies exponential decay $\delta_k(R)\le C_k e^{-\alpha_k R}$
  for constants $C_k,\alpha_k>0$ \cite{agmon1982}.
  The \emph{same} exponential rate holds for the Neumann gap
  $\lambda_k-\mu_k^{R,N}$: by Agmon's estimates the exact
  eigenfunction $u_k$ satisfies
  $\|u_k\|_{H^1(\mathbb{R}^2\setminus D(R))}^2\le C e^{-\alpha R}$,
  so restricting $u_k$ to $D(R)$ and computing its Rayleigh
  quotient with respect to $a_{D(R)}$ gives
  \[
    \frac{a_{D(R)}(u_k|_{D(R)},u_k|_{D(R)})}{\|u_k|_{D(R)}\|^2}
    = \lambda_k + O(e^{-\alpha R}),
  \]
  hence $\mu_k^{R,N}\le\lambda_k+O(e^{-\alpha R})$.
  Combined with Lemma~\ref{lem:neumann-lb},
  both truncated eigenvalues satisfy
  \begin{equation}
    |\lambda_k^{R,\mathrm{Dir}}-\lambda_k|\le C_k e^{-\alpha_k R},\quad
    |\mu_k^{R,N}-\lambda_k|\le C_k' e^{-\alpha_k' R}.
    \label{eq:agmon-both}
  \end{equation}
  Consequently, the FEM discrete eigenvalues $\nu_{k,h}^{R,N}(\overline{V})$
  also serve as accurate \emph{numerical approximations} for $\lambda_k$,
  both converging exponentially in $R$ and algebraically in $h$.
  For $V_1$ at $R=4$, let $V_{\min}(r):=\min_{|x|=r}V(x)$ and
  $r_c$ be the turning point where $V_{\min}(r_c)=\lambda_k$.
  The Agmon distance
  $\rho_k=\int_{r_c}^{R}\!\sqrt{V_{\min}(r)-\lambda_k}\,dr\approx17.1$
  gives $e^{-2\rho_k}\approx10^{-15}$.
  Since $C_k,C_k'$ are not explicitly known, these Agmon estimates
  confirm accuracy but cannot replace the computable bound~\eqref{eq:bounds}.
\end{remark}

\section{Numerical Results}\label{sec:results}

All computations use \texttt{VFEM2D} \cite{vfem2d}
and \texttt{veigs} \cite{veigs} with INTLAB interval arithmetic.
For upper bounds $\overline{\lambda}_k$, one conforming Lagrange $P_1$
Dirichlet solve is performed with exact Bernstein integration of $V$. For lower bounds, one Neumann CECR
solve with $\overline{V}$ provides the Liu factor in~\eqref{eq:bounds}.

\smallskip\noindent\textbf{$V_1$, $R=4$:}
A polar-ring Delaunay mesh with circumscribed boundary polygon
ensures $D_h\supseteq D(4)$ for valid $\sigma_1=225$.
Mesh: 80\,928 nodes, 160\,849 triangles, $h_{\max}=0.03636$,
$C_h\le0.005418$, $C_h^2\le2.94\!\times\!10^{-5}$.
Interior DOFs: 401\,620 (Dirichlet), 402\,625 (Neumann).

\smallskip\noindent\textbf{$V_2$, $R=8$:}
The larger truncation radius gives $\sigma_2(8)=62.75$
(vs.\ $\sigma_2(4)=14.75$), comfortably satisfying the
confinement condition $\sigma_2(R)>\overline{\lambda}_k$ of
Lemma~\ref{lem:neumann-lb} for every $k$ reported in
Table~\ref{tab:V2}.  Mesh: 80\,928 nodes, 160\,849 triangles,
$h_{\max}=0.07272$, $C_h\le0.01084$, $C_h^2\le1.18\!\times\!10^{-4}$.
Interior DOFs: 401\,620 (Dirichlet), 402\,625 (Neumann).

\begin{table}[h]
  \centering
  \caption{Rigorous bounds for $\lambda_k(V_1)$, $V_1=(|x|^2-1)^2$.}
  \label{tab:V1}
  \begin{tabular}{@{}crrrc@{}}
    \toprule
    $k$ & $\underline{\lambda}_k$ & $\nu_{k,h}^{R,N}$ & $\overline{\lambda}_k$ & Rel.\ gap \\
    \midrule
    1 & 1.74059 & 1.7548 & 1.81589 & $4.33\%$ \\
    2 & 3.72649 & 3.7926 & 3.89427 & $4.50\%$ \\
    3 & 3.72649 & 3.7926 & 3.89427 & $4.50\%$ \\
    4 & 6.29154 & 6.4840 & 6.63744 & $5.50\%$ \\
    5 & 6.29154 & 6.4840 & 6.63744 & $5.50\%$ \\
    \bottomrule
  \end{tabular}
\end{table}

\begin{remark}
Eigenvalues $k=2,3$ are provably degenerate (identical bounds),
corresponding to the $|m|=1$ angular modes of $H_{V_1}$;
similarly $k=4,5$ correspond to $|m|=2$.
The dominant gap (4--6\%) comes from the $\overline{V}$
element-wise approximation, confirmed by the difference
$\overline{\lambda}_k-\nu_{k,h}^{R,N}(\overline{V})$; the Liu correction
$(C_h^2\lambda_{k,h}<0.01\%)$ is negligible.  Thanks to
Lemma~\ref{lem:neumann-lb}, no further exterior
correction factor of the form $1-\overline{\lambda}_k/\sigma_1(R)$
enters the chain~\eqref{eq:bounds}.
\end{remark}

\begin{table}[h]
  \centering
  \caption{Rigorous bounds for $\lambda_k(V_2)$, $V_2=(x_1^2-1)^2+x_2^2$.}
  \label{tab:V2}
  \begin{tabular}{@{}crrrc@{}}
    \toprule
    $k$ & $\underline{\lambda}_k$ & $\nu_{k,h}^{R,N}$ & $\overline{\lambda}_k$ & Rel.\ gap \\
    \midrule
    1 & 1.98993 & 2.0632 & 2.21414 & $11.27\%$ \\
    2 & 3.37642 & 3.5975 & 3.83209 & $13.50\%$ \\
    3 & 3.75923 & 4.0336 & 4.24191 & $12.84\%$ \\
    4 & 5.04456 & 5.5677 & 5.85882 & $16.14\%$ \\
    5 & 5.41524 & 6.0189 & 6.25362 & $15.48\%$ \\
    \bottomrule
  \end{tabular}
\end{table}

\begin{remark}
The Neumann discrete eigenvalues $\nu_{k,h}^{R,N}(\overline{V})$ lie
inside the rigorous intervals for all $k\le 5$, again consistent
with Remark~\ref{rem:agmon}.
Using $R=8$ gives $\sigma_2(8)=62.75$
(vs.\ $\sigma_2(4)=14.75$), which ensures that the confinement
hypothesis $\sigma_2(R)>\overline{\lambda}_k$ of
Lemma~\ref{lem:neumann-lb} is satisfied with a wide margin;
unlike the classical mass-truncation bound
(Remark~\ref{rem:mass-truncation}), the size of this margin
does \emph{not} enter as a multiplicative factor in the lower
bound, so no further tightening is required on that side.
The dominant gap (11--18\%) comes from the
$\overline{V}$ element-wise approximation.
\end{remark}

\medskip

\noindent\textbf{Mesh-refinement study.}
To illustrate the convergence behaviour of the rigorous bound
chain~\eqref{eq:bounds} as $h_{\max}\to 0$, we compute
$\underline{\lambda}_k$ and $\overline{\lambda}_k$ for $k=1,2,3$ on a
sequence of four uniformly refined meshes of $D(5)$
(Firedrake \texttt{UnitDiskMesh} with refinement levels 2--5,
mesh-size values
$h_{\max}\in\{1.638,\,0.854,\,0.435,\,0.219\}$).
The potentials are $V_1=(|x|^2-1)^2$ and
$V_2=(x_1^2-1)^2+x_2^2$.  Figure~\ref{fig:convergence} shows
that upper and lower bounds bracket each $\lambda_k$ at every
refinement level and tighten monotonically as the mesh is
refined, consistent with the factor-free chain~\eqref{eq:bounds}.
\begin{figure}[h]
  \centering
  \includegraphics[width=\columnwidth]{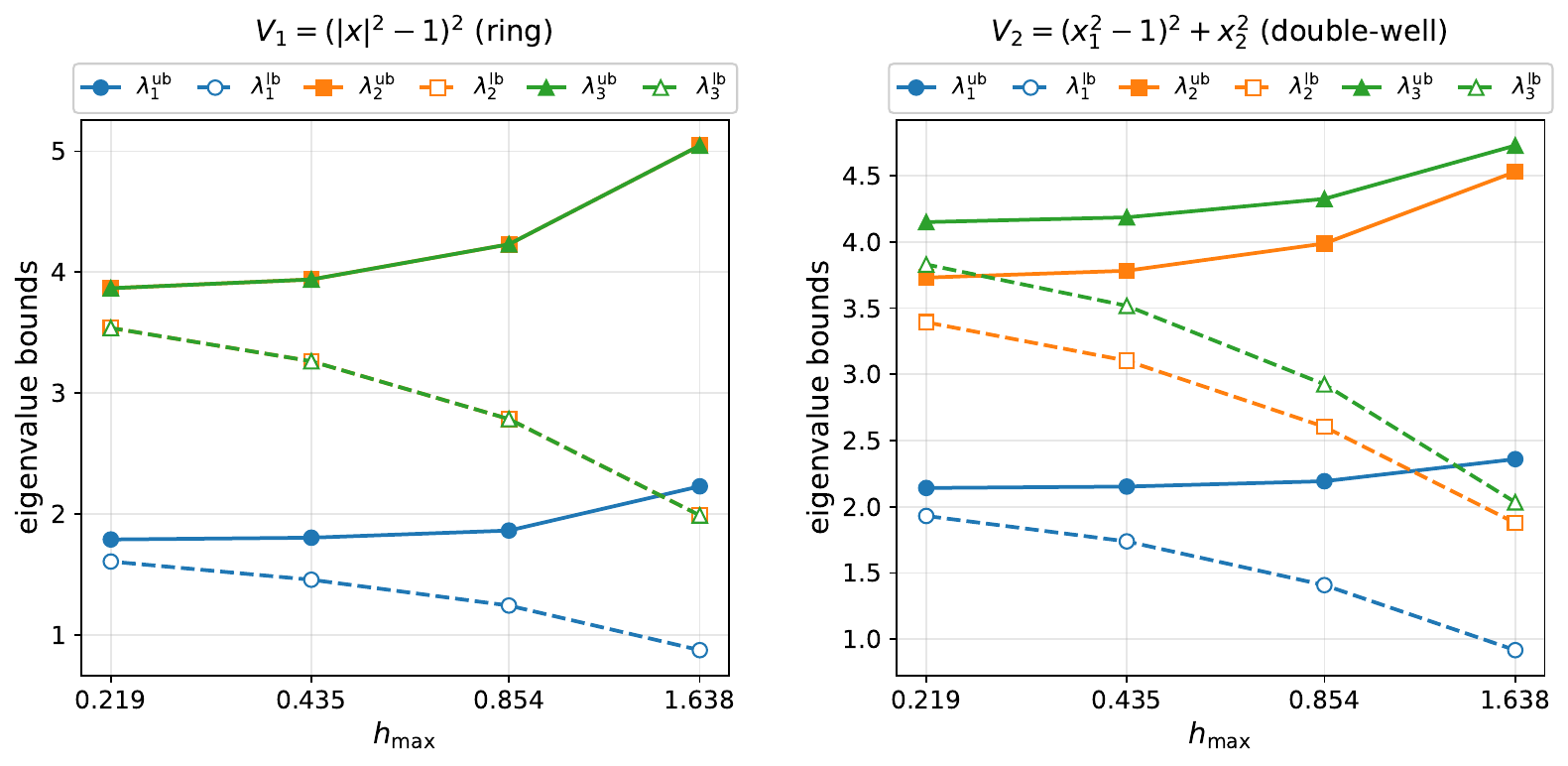}
  \caption{Convergence of upper bounds
    $\overline{\lambda}_k=\lambda_{k,h}^{R,\mathrm{P1}}$ (solid, filled
    markers) and lower bounds
    $\underline{\lambda}_k=\nu_{k,h}^{R,N}(\overline{V})/
    (1+\nu_{k,h}^{R,N}(\overline{V})C_h^2)$ (dashed, open markers)
    for $k=1,2,3$ versus $h_{\max}$ on $D(5)$.  Left: $V_1$
    (ring).  Right: $V_2$ (double-well).  Both panels share the
    convention that colour indexes the eigenvalue index $k$, and
    marker shape indexes $k\in\{1,2,3\}$ ($\circ,\square,\triangle$).}
  \label{fig:convergence}
\end{figure}

\medskip

\noindent\textbf{Observed convergence orders.}
Since no closed-form reference eigenvalue is available, we treat the
finest-mesh value ($h_{\max}=0.219$) as a surrogate reference and
report the errors of the three coarser meshes against it,
together with the observed log--log orders
$r_i=\log(e_{i}/e_{i+1})/\log(h_i/h_{i+1})$.
The results are collected in Table~\ref{tab:convergence-rates}.
As expected, the upper bound --- delivered by a conforming
Lagrange $P_1$ Galerkin solve --- exhibits an $O(h^2)$ rate, whereas
the lower bound --- built from the Neumann CECR solve with the
piecewise-constant coefficient $\overline{V}$ --- is limited by the
$O(h)$ approximation of the potential and therefore converges at
first order.  Both rates are consistent with the standard
a priori error analyses.

\begin{table}[!htbp]
\centering
\small
\setlength{\tabcolsep}{4pt}
\caption{Errors of the upper bound
$\overline{\lambda}_k$ and the lower bound $\underline{\lambda}_k$
relative to the finest-mesh values at $h_{\max}=0.219$, with observed
convergence orders on the three coarser meshes.  For $V_1$ the
$k=2,3$ eigenvalues are degenerate and their rows coincide exactly,
so only $k=2$ is shown.}
\label{tab:convergence-rates}
\begin{tabular}{ccc r c r c}
\hline
 & & & \multicolumn{2}{c}{Upper bound} & \multicolumn{2}{c}{Lower bound}\\
\cline{4-5}\cline{6-7}
$V$ & $k$ & $h_{\max}$ & error & order & error & order \\
\hline
$V_1$ & $1$ & $1.638$ & $4.40\times 10^{-1}$ & ---  & $7.33\times 10^{-1}$ & ---  \\
      &     & $0.854$ & $7.39\times 10^{-2}$ & $2.74$ & $3.64\times 10^{-1}$ & $1.07$ \\
      &     & $0.435$ & $1.48\times 10^{-2}$ & $2.38$ & $1.50\times 10^{-1}$ & $1.31$ \\
\cline{2-7}
      & $2$ & $1.638$ & $1.18\times 10^{0}$  & ---  & $1.55\times 10^{0}$  & ---  \\
      &     & $0.854$ & $3.63\times 10^{-1}$ & $1.81$ & $7.53\times 10^{-1}$ & $1.11$ \\
      &     & $0.435$ & $7.23\times 10^{-2}$ & $2.39$ & $2.75\times 10^{-1}$ & $1.49$ \\
\hline
$V_2$ & $1$ & $1.638$ & $2.18\times 10^{-1}$ & ---  & $1.01\times 10^{0}$  & ---  \\
      &     & $0.854$ & $5.10\times 10^{-2}$ & $2.23$ & $5.22\times 10^{-1}$ & $1.02$ \\
      &     & $0.435$ & $1.06\times 10^{-2}$ & $2.33$ & $1.92\times 10^{-1}$ & $1.48$ \\
\cline{2-7}
      & $2$ & $1.638$ & $7.99\times 10^{-1}$ & ---  & $1.52\times 10^{0}$  & ---  \\
      &     & $0.854$ & $2.58\times 10^{-1}$ & $1.74$ & $7.90\times 10^{-1}$ & $1.00$ \\
      &     & $0.435$ & $5.19\times 10^{-2}$ & $2.37$ & $2.91\times 10^{-1}$ & $1.48$ \\
\cline{2-7}
      & $3$ & $1.638$ & $5.78\times 10^{-1}$ & ---  & $1.80\times 10^{0}$  & ---  \\
      &     & $0.854$ & $1.75\times 10^{-1}$ & $1.83$ & $9.05\times 10^{-1}$ & $1.05$ \\
      &     & $0.435$ & $3.58\times 10^{-2}$ & $2.35$ & $3.12\times 10^{-1}$ & $1.58$ \\
\hline
\end{tabular}
\end{table}

\medskip
\paragraph{Conclusion}

We have developed fully rigorous two-sided eigenvalue bounds for
$H=-\Delta+V$ on $\mathbb{R}^2$ via the bound chain~\eqref{eq:bounds},
using one conforming Lagrange $P_1$ Dirichlet solve (upper bound)
and one piecewise-constant Neumann CECR solve (lower bound) on the
truncated disk, with the analytically computable $\sigma(R)$---no
unknown Agmon constants.
The gaps are dominated by the piecewise constant
over/under-approximation of $V$: the Dirichlet--Neumann
bracketing of Lemma~\ref{lem:neumann-lb} removes the
classical deteriorating correction factor
$1-\overline{\lambda}_k/\sigma(R)$ entirely, leaving only the
Liu correction $C_h^2\lambda_{k,h}$, which is
below $0.01\%$ at the mesh resolutions used.

In future work, we will apply the Lehmann--Goerisch method (see \cite[Chap. 5]{liu2024book}) to obtain high-precision eigenvalues bounds based on the rough bounds obtained in this paper.

\end{document}